\documentclass{article}
\usepackage{amsmath,amsthm}
\usepackage{ amssymb }
\usepackage{times}
\usepackage{fullpage}
\usepackage{dsfont}
\usepackage{graphicx}
\usepackage{epstopdf}
\usepackage{hyperref}
\usepackage{enumerate}
\usepackage{color}
\usepackage{caption}
\usepackage{subcaption}
\usepackage{siunitx}
\usepackage{mathrsfs}
\usepackage[alphabetic,lite,nobysame]{amsrefs}

\theoremstyle{plain} 
\newtheorem{thm}{Theorem}[section]
\newtheorem*{thm*}{Theorem}

\newtheorem{lem}[thm]{Lemma}
\newtheorem{prop}[thm]{Proposition}
\theoremstyle{definition}

\theoremstyle{remark}

\newcommand{\R}{\mathbb{R}}

\newcommand{\mcA}{\mathcal{A}}
\newcommand{\mcC}{\mathcal{C}}

\newcommand{\mcL}{\mathcal{L}}

\newcommand{\mcP}{\mathcal{P}}

\newcommand{\msP}{\mathscr{P}}

\newcommand{\diam}{\mathrm{diam}}
\newcommand{\lra}{\longrightarrow}

\newcommand{\fsubd}{\mathrel{{\scriptstyle\searrow}\kern-1ex^d\kern0.5ex}}
\newcommand{\bsubd}{\mathrel{{\scriptstyle\swarrow}\kern-1.6ex^d\kern0.8ex}}

\begin{document}

\title{Hyperbolicity constants for pants and relative pants graphs}
\author{Ashley Weber}
\date{ }

\maketitle

\begin{abstract}
The pants graph has proved to be influential in understanding 3-manifolds concretely.  
This stems from a quasi-isometry between the pants graph and the Teichm\"uller space with the Weil-Petersson metric.  
Currently, all estimates on the quasi-isometry constants are dependent on the surface in an undiscovered way.  
This paper starts effectivising some constants which begins the understanding how relevant constants change based on the surface.  We do this by studying the hyperbolicity constant of the pants graph for the five-punctured sphere and the twice punctured torus.  The hyperbolicity constant of the relative pants graph for complexity 3 surfaces is also calculated.  Note, for higher complexity surfaces, the pants graph is not hyperbolic or even strongly relatively hyperbolic.
\end{abstract}

\section{Introduction}
The pants graph has been instrumental in understanding Teichm\"uller space.  This is because the pants graph is quasi-isometric to Teichm\"uller space equipped with the Weil-Petersson metric \cite{Brock-WPtoPants}. Brock and Margalit used pants graphs to show that all isometries of Teichm\"uller space with the Weil-Petersson metric arise from the mapping class group of the surface \cite{BM-WPisom}.  This relationship was also used to classify for which surfaces the associated Teichm\"uller space is hyperbolic.  
The relationship between the pants graph and Teichm\"uller space has been used to study volumes of 3-manifolds \cite{Brock-WPtoPants, Brock-WPtrans}.  In particular, it has been used to relate volumes of the convex core of a hyperbolic 3-manifold to the distance of two points in Teichm\"uller space.  It has also related the volume of a hyperbolic 3-manifold arising from a psuedo-Anosov element in the mapping class group to the translation length of the psuedo-Anosov element as applied to the pants graph.  Both of these relations have constants which depend on the surface; this paper is the start of effectivising those constants.  Notice Aougab, Taylor, and Webb have some effective bounds on the quasi-isometry bounds, however even these still depend on the surface in a way that is unknown \cite{ATW}.

Let $S_{g,p}$ be a surface with genus $g$ and $p$ punctures.  We define the complexity of a surface to be $\xi(S_{g,p}) =3g + p - 3 $.   
Brock and Farb have shown that the pants graph is hyperbolic if and only if the complexity of the surface is less than or equal to $2$ \cite{BF}.  Brock and Masur showed that in a few cases the pants graph is strongly relatively hyperbolic, specifically when $\xi(S) = 3$ \cite{BM}.  Even though hyperbolicity is well studied for the pants graph, the hyperbolicity constants associated with the pants graph or the relative pants graph is not.  In addition to having a further understanding of the quasi-isometry mentioned above and all of its applications, actual hyperbolicity constants are useful in answering questions about asymptotic time complexity of certain algorithms, especially those involving the mapping class group.  
More speculatively, estimates on hyperbolicity constants may be crucial to effectively understand the virtual fibering conjecture, which relates the geometry of the fiber to the geometry of the base surface.  The focus of this paper is to find hyperbolicity constants for the pants graph and relative pants graph, when these graphs are hyperbolic.  

\begin{thm*}[\emph{c.f. Theorem ~\ref{main thm 1}}]
For a surface $S = S_{0,5}, S_{1,2}$, $\mcP(S)$ is $2,691,437$-thin hyperbolic.
\end{thm*}

Computing the asymptotic translation lengths of an element in the mapping class group on $\mcP(S)$ is a question explored by Irmer \cite{Irmer}.  Bell and Webb have an algorithm that answers this question for the curve graph \cite{BellWebb}.  Combining the works of Irmer, and Bell and Webb, one could conceivably come up with an algorithm for asymptotic translation lengths on $\mcP(S)$.  In this case, the above Theorem would put a bound on the run-time of the algorithm in the cases that $S = S_{0,5}, S_{1,2}$. 

We now turn our attention to the relatively hyperbolic cases.

\begin{thm*}[\emph{c.f. Theorem ~\ref{main thm 2}}]
For a surface $S = S_{3,0}, S_{1,3}, S_{0,6}$, $\mcP_{rel}(S)$ is $2,606,810,489$-thin hyperbolic.
\end{thm*}

To show both of our main theorems, we construct a family of paths that is very closely related to hierarchies, introduced in \cite{MMII}.  We show that this family of paths satisfies the thin triangle condition which, by a theorem of Bowditch, allows us to conclude the whole space is hyperbolic \cite{Bow}.  A key tool used throughout is the Bounded Geodesic Image Theorem \cite{MMII}.  This theorem allows us to control the length of geodesics in subspaces.  

This method cannot be made to generalize to pants graphs in general since any pants graph of a surface with complexity higher than $3$ is not strongly relatively hyperbolic \cite{BM}.  Although, this method may be able to be used for other graphs which are variants on the pants graph.

One might consider approaching this problem by finding the sectional curvature of Teichm\"uller space and using the quasi-isometry to inform on the hyperbolicity constant of the pants graph.  If the sectional curvature is bounded away from zero, one can relate the curvature of the space to the hyperbolicity constant of the space.  However, the sectional curvature of Teichm\"uller sapce is not bounded away from zero \cite{Huang}.  Therefore, this technique cannot be used. 

\textbf{Acknowledgments:}  I would like to thank my advisor, Jeff Brock, for suggesting this problem, support, and helpful conversations.  I'd also like to thank Tarik Aougab and Peihong Jiang for helpful conversations.  

\section{Preliminaries}

\subsection{Hyperbolicity}
Assume $\Gamma$ is a connected graph which we equip with the metric where each edge has length 1.  We give two definitions of a graph being hyperbolic.  A triangle in $\Gamma$ is $k$-\textit{centered} if there exists a vertex $c \in \Gamma$ such that $c$ is distance $\leq k$ from each of its three sides.  $\Gamma$ is $k$-\textit{centered hyperbolic} if all geodesic triangles (triangles whose edges are geodesics) are $k$-centered.  
We say a triangle in $\Gamma$ is $\delta$-\textit{thin} if each side of the triangle is contained in the $\delta$-neighborhood of the other two sides for some $\delta \in \R$.  A graph is $\delta$-\textit{thin hyperbolic} if all geodesic triangles are $\delta$-thin.  Note that $\delta$-thin hyperbolic and $k$-centered hyperbolic are equivalent up to a linear factor \cite{ABC}.  

\begin{lem}
\label{centered to thin}
If $\Gamma$ is $k$-centered hyperbolic then $\Gamma$ is $4k$-thin hyperbolic.
\end{lem}

The following proof is very similar to the proof of an existence of a global minsize of triangles implies slim triangles in \cite{ABC} (Proposition 2.1).
\begin{proof}
We denote $[a,b]$ as a geodesic between $a$ and $b$; if $c \in [a,b]$ then $[a, c]$ or $[c,b]$ refers to the subpath of $[a,b]$ with $c$ as one of the endpoints.  
Consider the triangle $xyz$ and assume it is $k$-centered.  Let $p$ be the centered point and $x'$ be the point on the edge $[y,z]$ closest to $p$.  Similarly define $y'$ and $z'$.  Suppose there is a point $t \in [x,z']$ such that $d(t, [x, y']) > 2k$.  Let $u$ be the point in $[t, z']$ nearest to $t$ such that $d(u, u') = 2k$ for some point $u' \in [x, y']$, see Figure \ref{center to thin figure}.  

Consider the geodesic triangle $uu'x$.  There exists points $a$, $b$, and $c$ on the three sides of $uu'x$ that are less than or equal to $k$ away from some point $q$, see Figure \ref{center to thin figure}.  Since $a \in [x, u]$, by assumption $a$ does not lie in $[t, u]$ and $d(u, a) \leq 4k$.  So $d(t, u') \leq 4k$ or $d(t, c) \leq 4k$, making the triangle $xyz$ $4k$-thin.
\end{proof}

\begin{figure}
\centering
\includegraphics[trim = {0, 4.3in, 0, 3.7in}, clip, scale = .4]{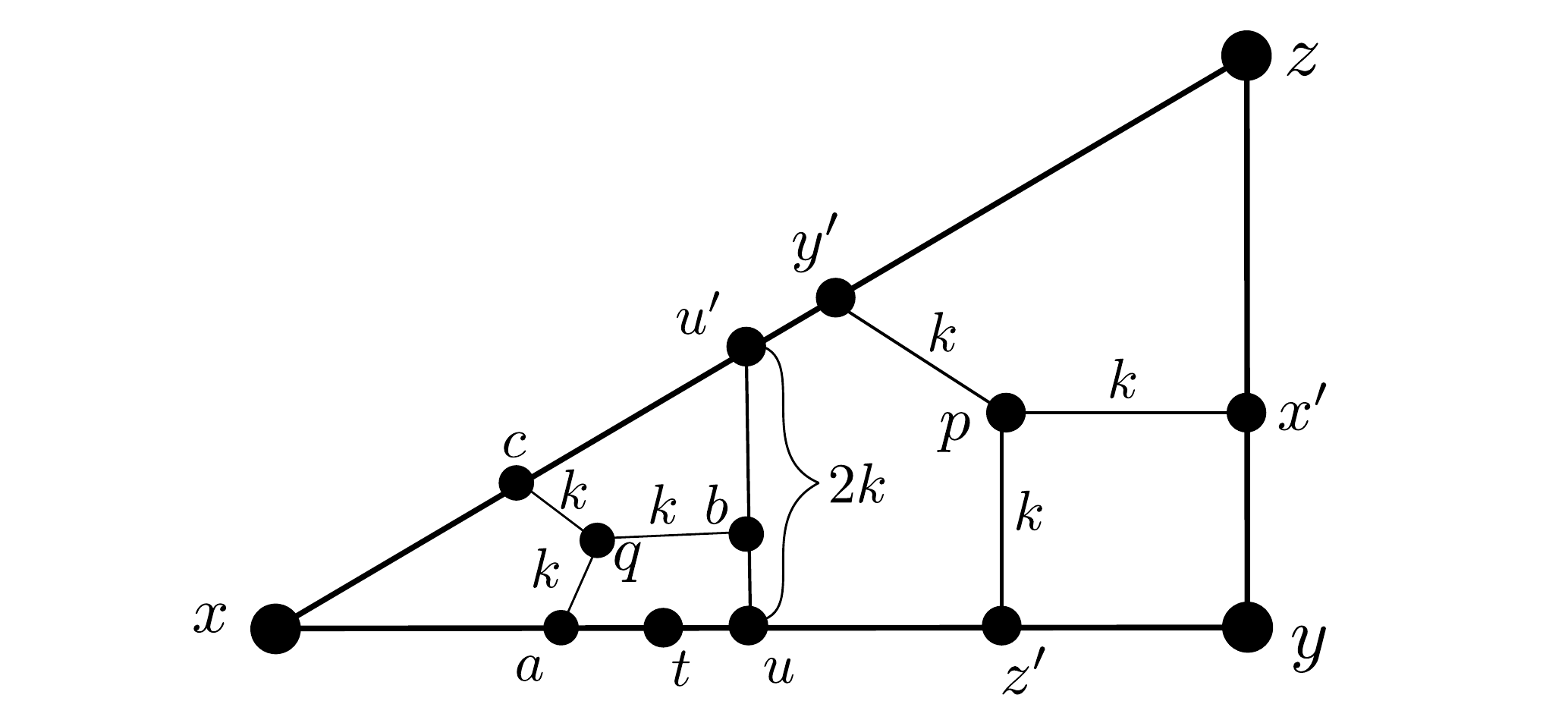}
\caption{}
\label{center to thin figure}
\end{figure}

Bowditch shows, in \cite{Bow} Proposition 3.1, that we don't always have to work with geodesic triangles to show hyperbolicity of a graph. 
\begin{prop}[\cite{Bow}]
\label{subset hyperbolic}
Given $h \geq 0$, there exists $\delta \geq 0$ with the following property.  Suppose that $G$ is a connected graph, and that for each $x, y \in V(G)$, we have associated a connected subgraph, $\mcL(x,y) \subset G$, with $x, y \in \mcL(x,y)$.  Suppose that: 
\begin{enumerate}
\item for all $x, y, z \in V(G)$, 
\begin{equation*}
\mcL(x,y) \subset N_h(\mcL(x,z) \cup \mcL(z, y))
\end{equation*}
and
\item for any $x, y \in V(G)$ with $d(x,y) \leq 1$, the diameter of $\mcL(x,y)$ in $G$ is at most $h$.
\end{enumerate}
Then $G$ is $\delta$-thin hyperbolic.  In fact, we can take any $\delta \geq (3m-10h)/2$, where $m$ is any positive real number satisfying 
\begin{equation*}
2h(6 + \log_2(m+2)) \leq m.
\end{equation*}
\end{prop}

\subsection{Graphs}

Let $S = S_{g,p}$ be a surface where $g$ is the genus and $p$ is the number of punctures. We define $\xi(S_{g,p}) = 3g + p -3$ and refer to $\xi(S_{g,p})$ as the complexity of $S_{g,p}$.  
When $\xi(S) > 1$ the curve graph of $S$, $\mcC(S)$, originally introduced by Harvey in \cite{Harvey}, is a graph whose vertices are homotopy classes of essential simple closed curves on $S$ and there is an edge between two vertices if the curves can be realized disjointly, up to isotopy.  From here on when we talk about curves we really mean a representative of the homotopy class of an essential, non-peripheral, simple closed curve.  When $\xi(S) = 1$, the definition of the curve graph is slightly altered in order to have a non-trivial graph: the vertices have the same definition, but there is an edge between two curves if they have minimal intersection number.  We can similarly define the \textit{arc and curve graph}, $\mcA\mcC(S)$, where a vertex is either a homotopy class of curves or homotopy class of arcs and the edges represent disjointness.  This definition is the same for all surfaces such that $\xi(S) > 0$.  

A related graph associated to a surface is the pants graph.  We call a maximal set of disjoint curves on a surface a \textit{pants decomposition}.  For $\xi(S) \geq 1$ the \textit{pants graph}, denoted $\mcP(S)$, of a surface $S$ is a graph whose vertices are homotopy classes of pants decompositions and there exists an edge between two pants decompositions if they are related by an elementary move.  Pants decompositions $\alpha$ and $\beta$ differ by an elementary move if one curve, $c$, from $\alpha$ can be deleted and replaced by a curve that intersects $c$ minimally to obtain $\beta$, see Figure \ref{elementary moves}.

We equip both graphs with the metric where each edge is length 1.  Then $\mcC(S)$ and $\mcP(S)$ are complete geodesic metric spaces.

\begin{figure}
\centering
\includegraphics[trim = {0, 6.3in, 0, 4in}, clip, scale = .7]{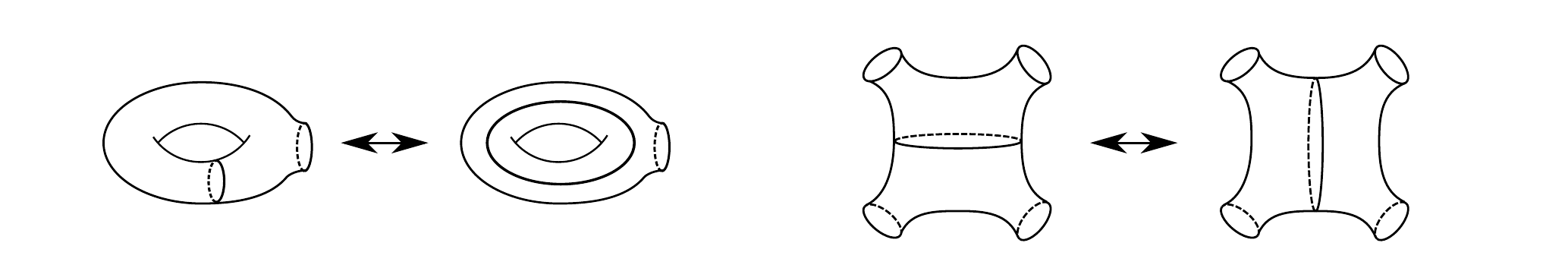}
\caption{These are the two elementary moves that form edges in $\mcP(S)$.}
\label{elementary moves}
\end{figure}

The hyperbolicity of these graphs have been studied before.

\begin{thm}[\cite{HPW}]
\label{curve hyp}
For any hyperbolic surface $S$, $\mcC(S)$ is $17$-centered hyperbolic.
\end{thm}

Brock and Farb showed:
\begin{thm}[\cite{BF}]
For any hyperbolic surface $S$, $\mcP(S)$ is hyperbolic if and only if $\xi(S) \leq 2$.
\end{thm}

\subsection{Relative graphs}

Let $S$ be a hyperbolic surface such that $\xi(S) \geq 3$.  We say that a curve $c \in \mcC(S)$ is \textit{domain separating} if $S \backslash c$ has two components of positive complexity.  Each domain separating curve $c$ determines a set in $\mcP(S)$, $X_c = \{\alpha \in \mcP(S)  | c \in \alpha \}$.  To form the \textit{relative pants graph}, denoted $\mcP_{rel}(S)$, we add a point $p_c$ for each domain separating curve and an edge from $p_c$ to each vertex in $X_c$, where each edge has length $1$.  Effectively, we have made the set $X_c$ have diameter $2$ in the relative pants graph.  

Brock and Masur have shown:

\begin{thm}[\cite{BM}]
For $S$ such that $\xi(S) = 3$, $\mcP_{rel}(S)$ is hyperbolic.  
\end{thm}

\subsection{Paths in the Pants Graph}

Here we describe how we will get a path in $\mcP(S)$ if $\xi(S) =2$ or $\mcP_{rel}(S)$ if $\xi(S) = 3$.  The paths for $\mcP(S)$ are hierarchies and were originally introduced by Masur and Minsky in \cite{MMII} (in more generality than we will use here); the paths in $\mcP_{rel}(S)$ are motivated by hierarchies.

Take two pants decompositions, $\alpha = \{ \alpha_0, \alpha_1\}$ and $\beta = \{ \beta_0, \beta_1\}$, in $\mcP(S)$ where $S = S_{0,5}$ or $S_{1,2}$.  To create a hierarchy between $\alpha$ and $\beta$ first connect $\alpha_0$ and $\beta_0$ with a geodesic path in $\mcC(S)$.  
This geodesic is referred to as the \textit{main geodesic}, $g_{\alpha\beta} = \{ \alpha_0 = g_0, \ldots, g_n = \beta_0\}$. For each $g_i$, $0 \leq i \leq n$, connect $g_{i-1}$ to $g_{i+1}$ by a geodesic, $\gamma_i$, in $\mcC(S\backslash g_i)$, where $g_{-1} = \alpha_1$ and $g_{n+1} = \beta_1$.  
The collection of all of these geodesics is a \textit{hierarchy} between $\alpha$ and $\beta$, generally pictured as in Figure \ref{Hierarchy picture}.  We often refer to the geodesic $\gamma_i$ as the geodesics whose domain is $\mcC(S \backslash g_i)$ or the geodesic connecting $g_{i-1}$ and $g_{i+1}$. 
We can turn a hierarchy into a path in $\mcP(S)$ by looked at all edges in turn, as pictured in Figure \ref{Hierarchy picture}.  We will often blur the line between the hierarchy being a path in the pants graph or a collection of geodesics - and refer to both as the hierarchy between $\alpha$ and $\beta$.    

\begin{figure}
\centering
\includegraphics[trim = {0, 6.5in, 0, 1.5in}, clip, scale = .6]{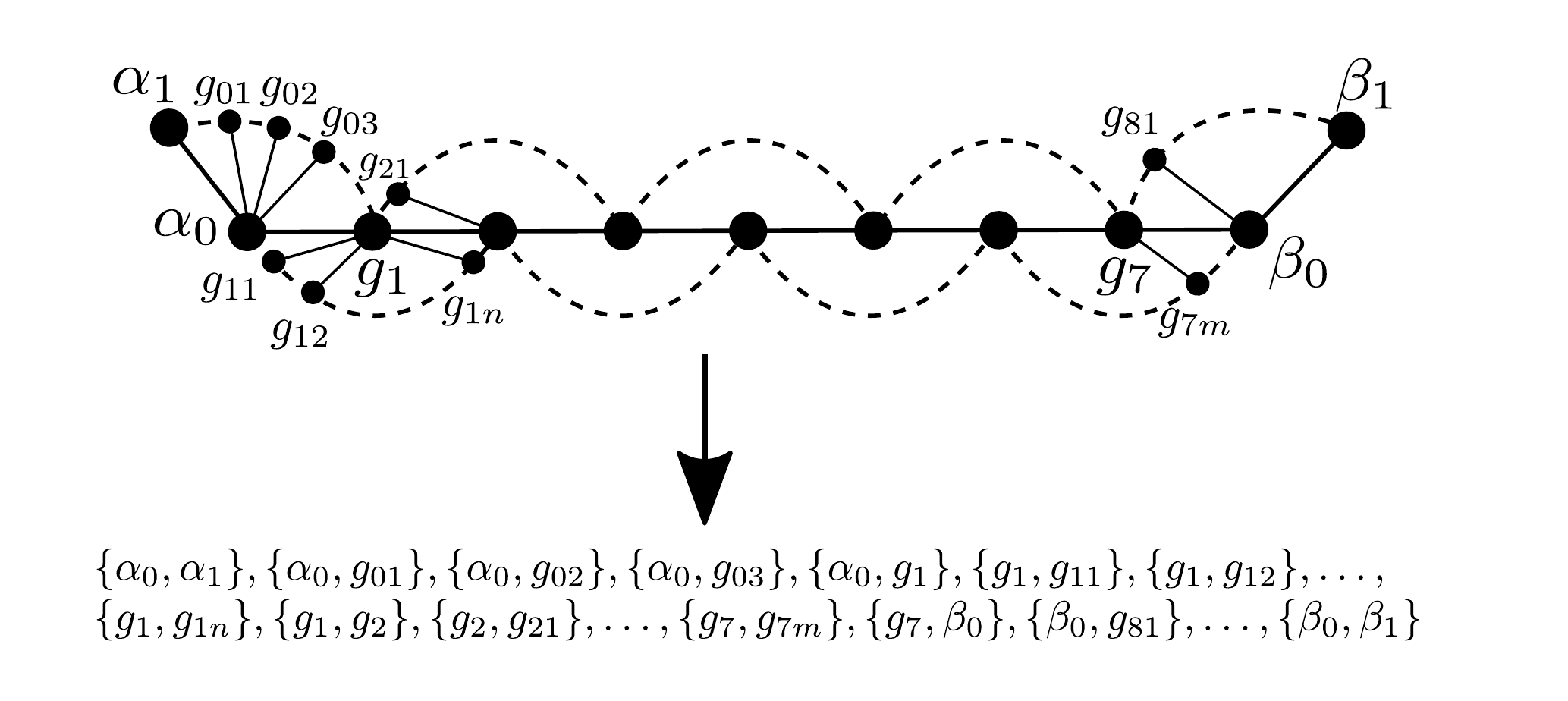}
\caption{What a hierarchy looks like in $S_{0,5}$ or $S_{1,2}$. Each edge represents a pants decomposition. The bottom gives the path in $\mcP(S)$ the hierarchy makes.}
\label{Hierarchy picture}
\end{figure}

\begin{figure}
\centering
\includegraphics[trim = {0, 5.5in, 0, 1.1in}, clip, scale = .7]{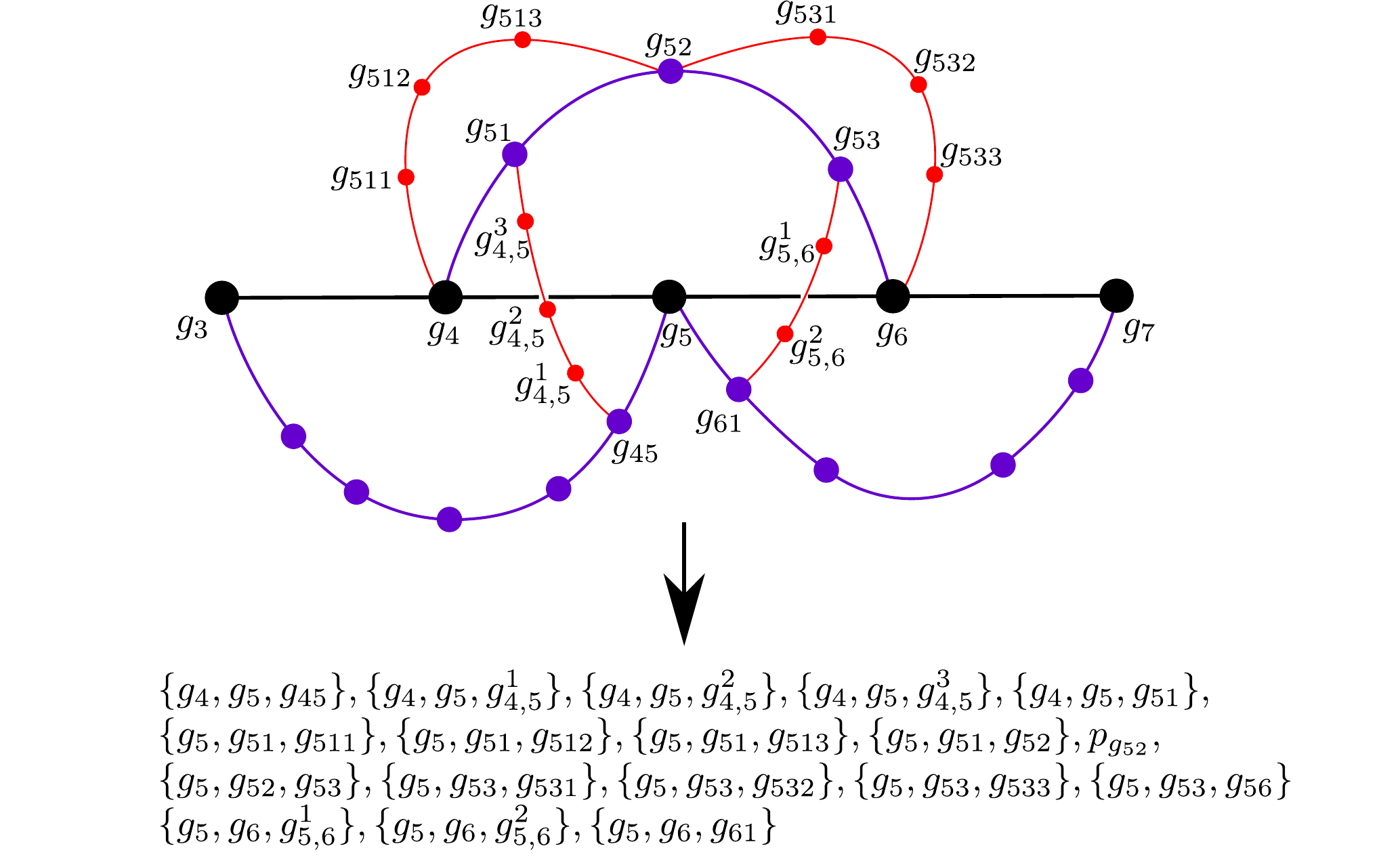}
\caption{The top represents part of a relative 3-archy for a surface with complexity 3.  Here we assume $g_{52}$ is domain separating and all other curves are non-domain separating.  The bottom gives the path in $\mcP(S)$ the relative 3-archy makes; it is the part of the path that contains $g_5$.}
\label{general hierarchy}
\end{figure}

Let $\xi(S) =3$.  We make a path in $\mcP_{rel}(S)$ using a similar technique. Take two pants decompositions in $\mcP_{rel}(S)$, $\alpha = \{\alpha_0, \alpha_1, \alpha_2\}$ and $\beta = \{\beta_0, \beta_1, \beta_2\}$.  Connect $\alpha_0$ to $\beta_0$ with a geodesic $g_{\alpha\beta}$ in $\mcC(S)$, we still refer to this as the main geodesic.  For every non-domain separating curve $w \in g$, connect $w^{-1}$ to $w^{+1}$ with a geodesic, $h$, in $\mcC(S \backslash w)$ where $w^{-1}$ and $w^{+1}$ are the curves before and after $w$ in $g$.  If $w = \alpha_0$ then $w^{-1} = \alpha_1$ and if $w = \beta_0$ then $w^{+} = \beta_1$.  Now for each non-domain separating curve $z \in h$ connect $z^{-1}$ to $z^{+1}$ with a geodesic in $\mcC(S \backslash (w \cup z))$, where $z^{-1}$ and $z^{+1}$ are the curves before and after $z$ in $h$.  If $z = w^{-1}$ then $z^{-1}$ is the curve preceding $w$ in the geodesic whose domain is $\mcC(S \backslash w^{-1})$. If $z = w^{+1}$ then $z^{+1}$ is the curve following $w$ in the geodesic whose domain is $\mcC(S \backslash w^{+1})$ (see Figure \ref{general hierarchy} (top)).  

We can get a path in $\mcP_{rel}(S)$ by a similar process as before - going along each of the edges.  Whenever we come across a domain separating curve, $c$, where $c$ is in the main geodesic or in a geodesic whose domain is $\mcC(S \backslash w)$ where $w$ is in the main geodesic, we add in the point $p_c$ into the path before moving on.  For an example see Figure \ref{general hierarchy}.  These paths are \textit{relative 3-archies}.  
As before, we will blur the line between the collection of geodesics and the path of a relative 3-archy.

When discussing hierarchies (or relative 3-archies), subsurface projections of curves or geodesics are involved.  The following maps are to define what is meant by subsurface projections \cite{MMII}. An \textit{essential subsurface} is a subsurface where each boundary component is essential. 

Let $\msP(X)$ be the set of subsets of $X$.  For a set $A$ we define $f(A) = \cup_{a \in A}f(a)$, for any map $f$. Take an essential, non-annular subsurface $Y \subset S$.  We define a map 
\begin{equation*}
\phi_Y: \mcC(S) \lra \msP(\mcA\mcC(Y))
\end{equation*} 
such that $\phi_Y(a)$ is the set of arcs and curves obtained from $a \cap Y$ when $\partial Y$ and $a$ are in minimal position.
Define another map
\begin{equation*}
\psi_Y : \msP(\mcA\mcC(Y)) \lra \msP(\mcC(Y))
\end{equation*}
such that if $a$ is a curve, then $\psi_Y(a) = a$, and if $b$ is an arc, then $\psi_Y(b)$ is the union of the non-trivial components of the regular neighborhood of $(b\cap Y) \cup \partial Y$ (see Figure \ref{nbhd}).  

\begin{figure}
\centering
\includegraphics[trim = {0, 2in, 0, 7.2in}, clip, scale = .5]{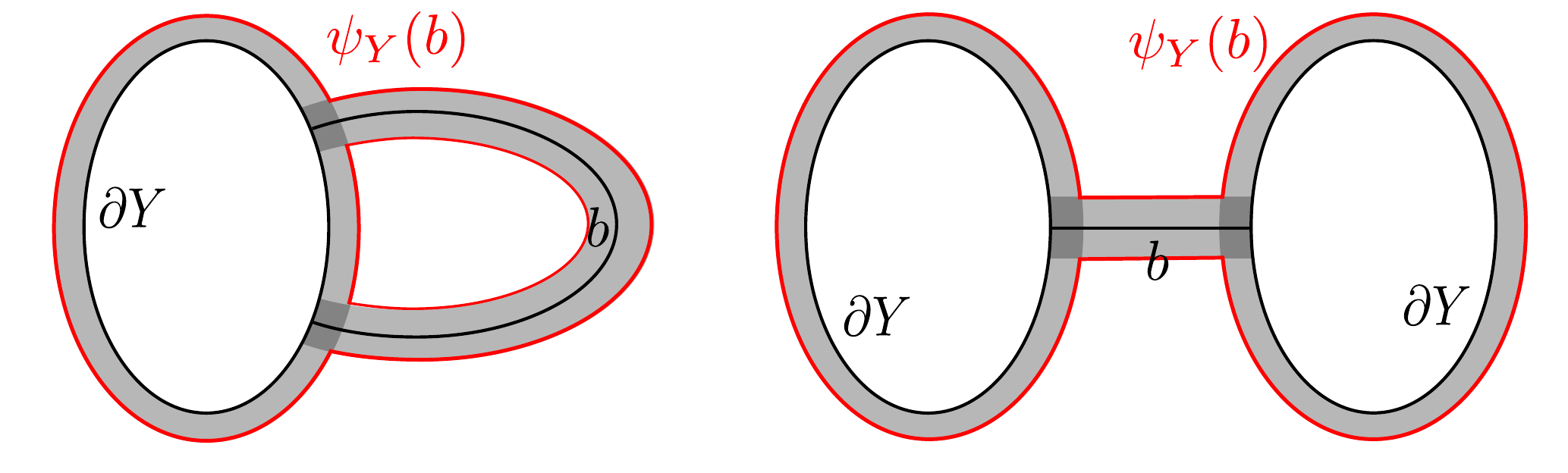}
\caption{The neighborhood of $(b \cap Y) \cup \partial Y$ is shaded with $\psi_Y(b)$ outlined in red.}
\label{nbhd}
\end{figure}

Composing these two maps we define the map 
\begin{align*}
\pi_Y: \mcC(S) &\lra \msP(\mcC(Y)) \\
c &\longmapsto \psi_Y(\phi_Y(c))
\end{align*}
We use this map to define distances in a subsurface: for any two sets $A$ and $B$ in $\mcC(S)$,
\begin{equation*}
d_Y(A, B) = d_Y(\pi_Y(A), \pi_Y(B)).
\end{equation*}
We often refer to this as the distance in the subsurface $Y$. 

The relationship between hierarchies and these maps give rise to some useful properties including the Bounded Geodesic Image Theorem which was originally proven by Masur-Minsky \cite{MMII}.  

\begin{thm}[Bounded Geodesic Image Theorem]
\label{bounded geodesic image}
Let $Y$ be a subsurface of $S$ with $\xi(Y) \neq 3$ and let $g$ be a geodesic segment, ray, or biinfinite line in $\mcC(S)$, such that $\pi_Y(v) \neq \emptyset$ for every vertex of $v$ of $g$.  There is a constant $M$ depending only on $\xi(S)$ such that $$\diam_Y(g) \leq M.$$
\end{thm}

It can be shown that $M$ is at most $100$ for all surfaces \cite{Webb}.


\section {Hyperbolicity of Pants Graph for Complexity 2}

In this section we explore the hyperbolicity constant for the pants graph of surfaces with complexity $2$.  Before we state any results, some notation must be discussed.  
Throughout the paper we denote $[a, b]_\Sigma$ as a geodesic in $\mcC(\Sigma)$ connecting $a$ to $b$, for any surface $\Sigma$.  If a geodesic satisfying this is contained in a hierarchy (or relative 3-archy, in later sections) being discussed, $[a,b]_\Sigma$ denotes the geodesic in the hierarchy.

\begin{thm}
\label{hierarchy k-centered}
For $S = S_{0,5}, S_{1,2}$, hierarchy triangles in $\mcP(S)$ are $8,900$-centered.
\end{thm}

\begin{proof}
Let $S = S_{0,5}$ or $S_{1,2}$.  Take three pants decompositions $\alpha = \{\alpha_0, \alpha_1\}$, $\beta = \{\beta_0, \beta_1\}$, and $\gamma = \{\gamma_0, \gamma_1\}$ in $S$.  Consider the triangle $\alpha\beta\gamma$ in $\mcP(S)$ where the edges are taken to be hierarchies instead of geodesics.  There are three cases:

\begin{enumerate}
\item All three main geodesics have a curve in common. 
\item Any two of the main geodesics share a curve, but not the third.
\item None of the main geodesics have common curves.
\end{enumerate}

In all three cases we will find a pants decomposition such that the hierarchy connecting this pants decomposition to each edge in $\alpha\beta\gamma$ is less than $8,900$.

\vspace{1ex}

\noindent \textbf{Case 1}:  Assume the main geodesics of all three edges share the curve $v \in \mcC(S)$.  Define $v_{\alpha \beta}^{-1}$ to be the curve on $g_{\alpha \beta}$ preceding $v$ and $v_{\alpha \beta}^{+1}$ the curve on $g_{\alpha \beta}$ following $v$ when viewing $g_{\alpha\beta}$ going from $\alpha_0$ to $\beta_0$.  
Similarly define $v_{\alpha \gamma}^{-1}$, $v_{\alpha \gamma}^{+1}$, $v_{\beta \gamma}^{-1}$, and $v_{\beta \gamma}^{+1}$.  See Figure \ref{Case 1}.

\begin{figure}
\centering
\includegraphics[trim = {0, 4.7in, 0, 1in}, clip, scale = .3]{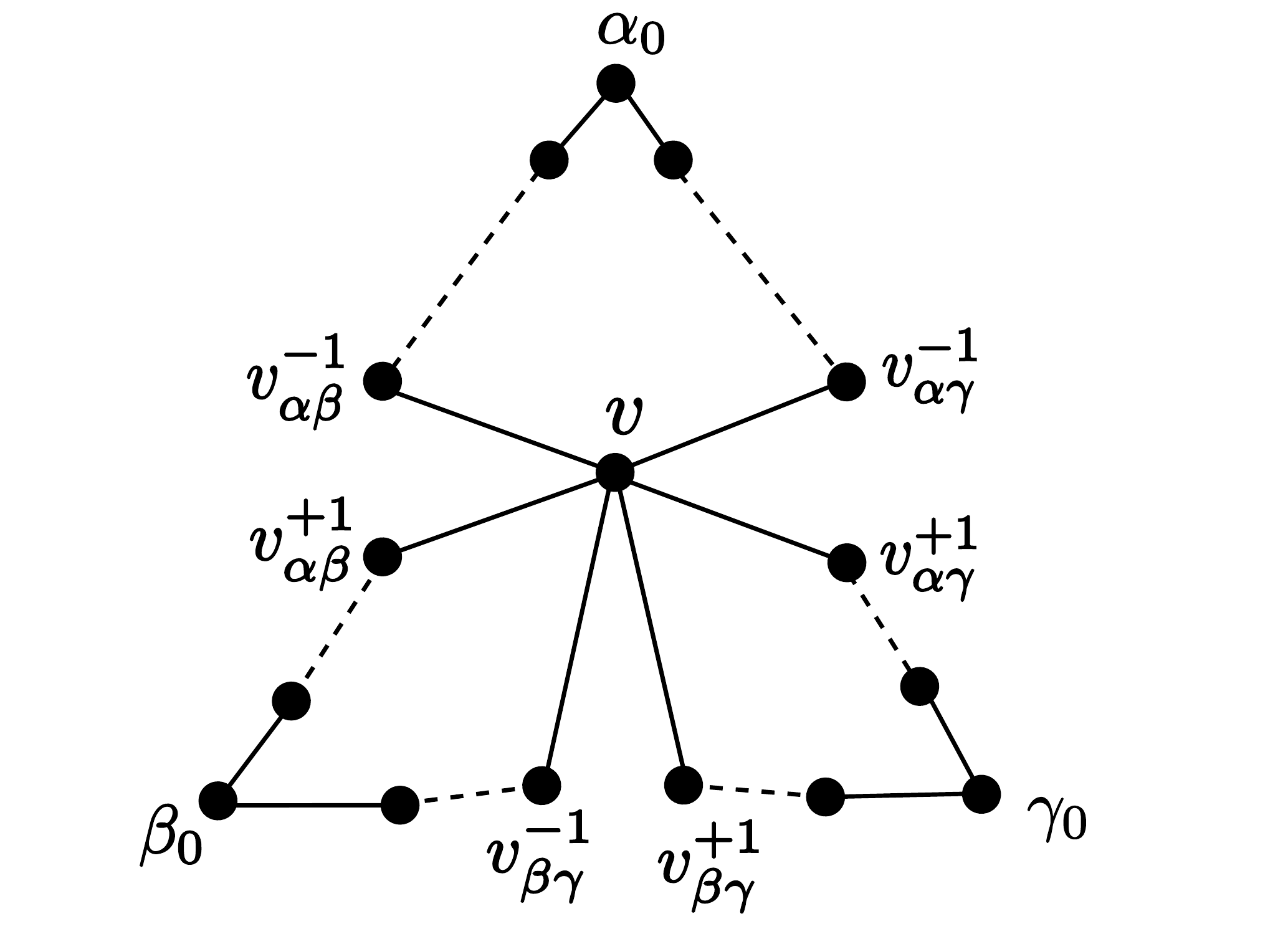}
\caption{Main geodesics of the hierarchy triangle in Case 1.}
\label{Case 1}
\end{figure}

We want to show the geodesics connecting $v_*^{-1}$ to $v_*^{+1}$ in $\mcC(S \backslash v)$ are not too far apart in $\mcC(S \backslash v)$. 
Connect $v_{\alpha\beta}^{-1}$ to $v_{\alpha\gamma}^{-1}$, $v_{\alpha\gamma}^{+1}$ to $v_{\beta \gamma}^{+1}$ and $v_{\beta\gamma}^{-1}$ to $v_{\alpha\beta}^{+1}$ by geodesics in $\mcC(S \backslash v)$.  We now have a loop in $\mcC(S\backslash v)$.  
Since all curves besides $v$ in $S$ intersect the subsurface $S \backslash v$ non-trivially we can apply the Bounded Geodesic Image Theorem on $[v_{\alpha \beta}^{-1}, \alpha_1]_S$ and $[\alpha_1, v_{\alpha \gamma}^{-1}]_S$ to get $d_{\mcC(S\backslash v)}(v_{\alpha \beta}^{-1}, v_{\alpha \gamma}^{-1}) \leq 2M$.  Similarly, $d_{\mcC(S\backslash v)}(v_{\alpha \gamma}^{+1}, v_{\beta \gamma}^{+1}) \leq 2M$ and $d_{\mcC(S\backslash v)}(v_{\beta \gamma}^{-1}, v_{\alpha\beta}^{+1}) \leq 2M$. 

Consider the geodesic triangle $v_{\alpha\beta}^{+1}v_{\alpha\gamma}^{-1}v_{\beta\gamma}^{+1}$ in $\mcC(S \backslash v)$.  We now have the picture in $\mcC(S \backslash v)$ as in Figure \ref{2 links are thin}.  By Theorem \ref{curve hyp}, the inner triangle is $17$ centered, call this center $z$.  Combining Theorem \ref{curve hyp} and Lemma \ref{centered to thin}, the outer three triangles are $17*4$-thin.  Therefore $z$ is at most $17*5 + 2M = 285$ away from each of the geodesics in the hierarchy triangle $\alpha\beta\gamma$ whose domain is $\mcC(S \backslash v)$. 

This all implies that $\alpha\beta\gamma$ is 285-centered at $\{v, z\}$.  

\begin{figure}
\centering
\includegraphics[trim = {0, 3.4in, 0, 4in}, clip, scale = .5]{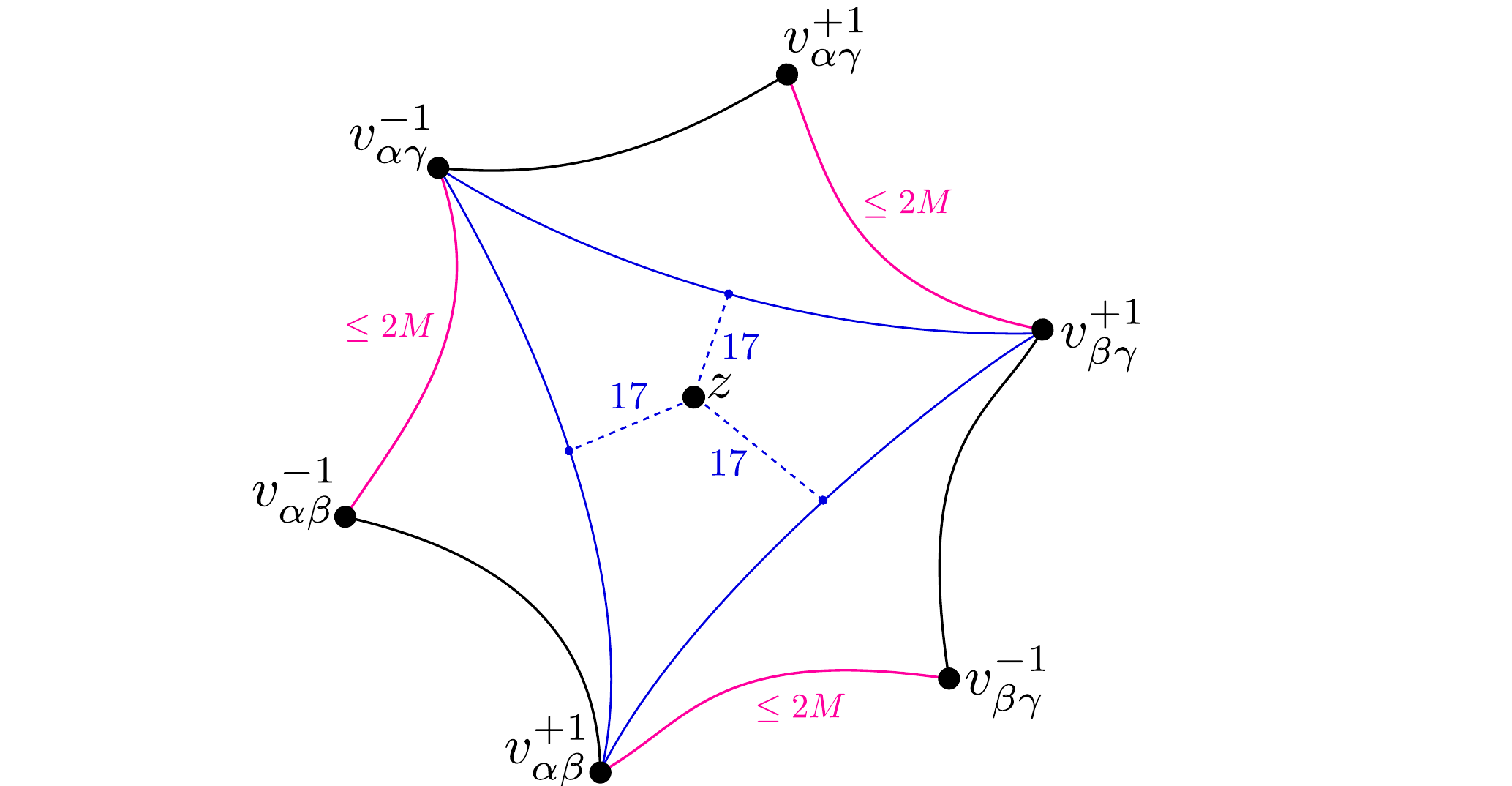}
\caption{Here all lines are geodesics in $\mcC(S \backslash v)$.  The black solid lines come from the geodesics in $\mcC(S \backslash v)$ in our original hierarchy triangle $\alpha\beta\gamma$.  The pink lines are the geodesics we added, and the blue lines make the inner triangle (which we also added).}
\label{2 links are thin}
\end{figure}

\vspace{1ex}

\noindent \textbf{Case 2}: Assume that at least two main geodesics share a common curve, but there is no point that all three main geodesics share the same curve.  First assume there is only one such shared curve.  Without loss of generality assume that $g_{\alpha\beta}$ and $g_{\alpha \gamma}$ share the curve $v$.  Then we can consider a new triangle with the main geodesics forming the triangle $v\beta_1\gamma_1$, see Figure \ref{Case 2}.  This new triangle has no shared curves so is covered by Case 3.

Now assume there is more than one shared curve between the main geodesics.  By definition of a geodesic, for any two main geodesics that share multiple curves, those curves have to show up in each main geodesic in the same order from either end, therefore we can just take the inner triangle where the edges share no curves and apply Case 3.

\begin{figure}
\centering
\includegraphics[trim = {0, 4in, 0, 1.2in}, clip, scale = .3]{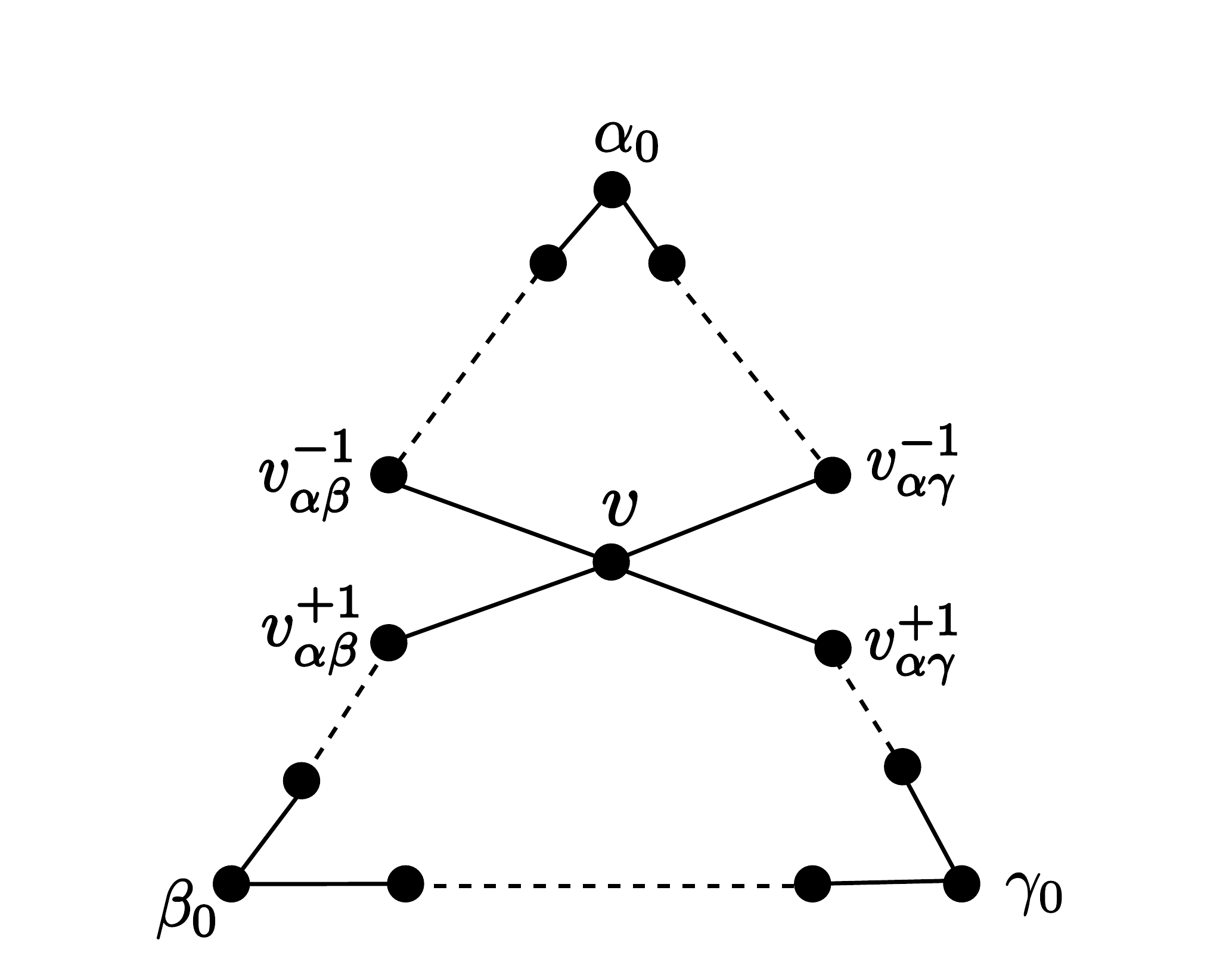}
\caption{Main geodesics of the hierarchy triangle in Case 2.}
\label{Case 2}
\end{figure}

\vspace{1ex}

\noindent \textbf{Case 3}: The argument given for this case is similar to the short cut argument in \cite{MMII}.  Assume none of the three main geodesics, $g_{\alpha \beta}, g_{\alpha \delta}$, and $g_{ \beta \delta}$ share a curve.  By Theorem \ref{curve hyp} there exists a curve $c \in \mcC(S)$ that is distance at most $17$ from $g_{\alpha \beta}, g_{\alpha \gamma}$, and $g_{ \beta \gamma}$; let $c$ be the curve that minimizes the distance from all three main geodesics.  Define $v_{\alpha \beta}$ to be the vertex in $g_{\alpha \beta}$ which has the least distance to $c$, and similarly define $v_{\alpha \gamma}$ and $v_{\beta \gamma}$.   

Consider the geodesic $[v_{\alpha\beta}, c]_S$ and let $c_0$ be the curve adjacent to $c$ in this geodesic.  Let $v_*^{-1}$ be the curve in $g_*$ that precedes $v_{*}$.  Now connect $\{v_{\beta\gamma}, v_{\beta\gamma}^{-1}\}$ to $\{c, c_0 \}$ with a hierarchy.  We denote the main geodesic of this hierarchy as $[c, v_{\beta\gamma}]_S$.  

Take a vertex $w \in [c, v_{ \beta \gamma}]_S$ where $w$ is not equal to $c$ or $v_{\beta\gamma}$ and let $w^{-1}$ and $w^{+1}$ denote the vertices directly before and after $w$ in $[c, v_{ \beta \gamma}]_S$.  
We want to show that the link connecting $w^{-1}$ to $w^{+1}$ in $S\backslash w$ is at most $5M$.  Assume $d_{S \backslash w} (w^{-1}, w^{+1}) \geq 5M$.  
Consider the path $[w^{+1}, v_{\beta\gamma}]_S \cup [v_{ \beta \gamma}, \beta_0]_S \cup [\beta_0, v_{\alpha \beta}]_S \cup [v_{\alpha \beta}, c]_S \cup [c, w^{-1}]_S$, where geodesics are taken to be on $g_*$ where appropriate.  The Bounded Geodesic Image Theorem, and our assumption that $d_{S \backslash w} (w^{-1}, w^{+1}) \geq 5M$, implies that $w$ must be somewhere on the path.  $w$ cannot be in $[w^{+1}, v_{\beta\gamma}]_S$, $[v_{ \beta \gamma}, \beta_0]_S$, or $[c, w^{-1}]_S$ since that would contradict the fact that they are geodesics or the definition of how we chose $c$ and $v_{\beta\gamma}$.  
Therefore, $w$ is in $[\beta_0, v_{\alpha \beta}]_S$ or $[v_{\alpha \beta}, c]_S$. 
 Without loss of generality assume $w \in [\beta_0, v_{\alpha \beta}]_S$.  We can apply the same logic to the path $[w^{+1}, v_{\beta\gamma}]_S \cup [v_{ \beta \gamma}, \gamma_0]_S \cup [\gamma_0, v_{\alpha \gamma}]_S \cup [v_{\alpha \gamma}, c]_S \cup [c, w^{-1}]_S$.  Now $w$ has to be in $[v_{\alpha \gamma}, c]_S$ so that it doesn't contradict the fact that the three main geodesic of the triangle $\alpha\beta\gamma$ do not share any curves.  However, now all three main geodesics are closer to $w$ than $c$, which contradicts our choice of $c$.  Therefore, the length of $[w^{-1}, w^{+1}]_{S \backslash w}$ is at most $5M$.  

\begin{figure}
\centering
\includegraphics[trim = {0, 4.1in, 0, 2.3in}, clip, scale = .5]{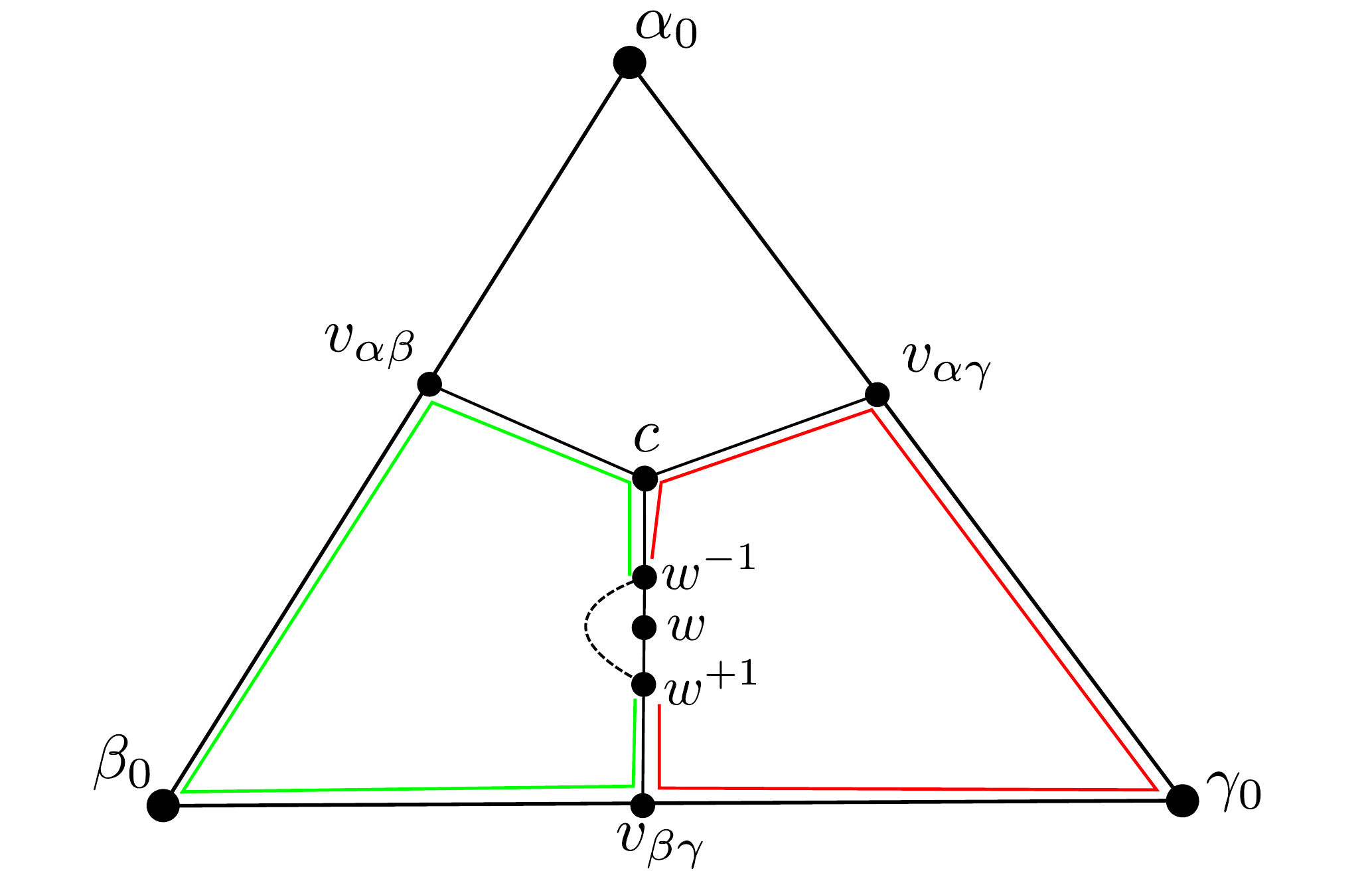}
\caption{Pictured here are the main geodesics in the triangle $\alpha\beta\gamma$ as in case 3.  The colored paths are the ones we consider when showing that the geodesic between $w^{-1}$ and $w^{+1}$ in $\mcC(S \backslash w)$ has length no more than $5M$.  }
\label{path for contradiction}
\end{figure} 


Using a similar argument we can show the geodesic in $\mcC(S \backslash v_{\beta\gamma})$ connecting $v_{\beta\gamma}^{-1}$ to the appropriate vertex in $[c, v_{\beta\gamma}]_S$ is $\leq 5M$.  
Now consider the geodesic in $\mcC(S \backslash c)$ connecting $c_0$ to the second vertex, $x$, of $[c, v_{\beta\gamma}]_S$.  Consider the path $[x, v_{\beta\gamma}]_S \cup [v_{\beta\gamma}, \beta_0]_S \cup [\beta_0, v_{\alpha\beta}]_S \cup [v_{\alpha\beta}, c_0]_S$.  $c$ cannot be in anywhere in this path, otherwise it would contradict how we chose $c$ or $v_*$.  So we can apply the Bounded Geodesic Image Theorem and get that $d_{S \backslash c}(c', x) \leq 4M$.  
Therefore the path from $\{v_{\beta\gamma}, v_{\beta\gamma}^{-1}\}$ to $\{c, c_0\}$ in the pants graph is less than or equal to $16(5M) + 5M + 4M$.  A similar argument can be made for the other two sides of the triangle $\alpha\beta\gamma$, so $\{c, c_0\}$ can be taken to be a center of the triangle.  Since $M \leq 100$ the triangle $\alpha\beta\gamma$ is $8,900$-centered at $\{c, c_0\}$.
\end{proof}

\begin{thm}
\label{main thm 1}
For a surface $S = S_{0,5}, S_{1,2}$, $\mcP(S)$ is $2,691,437$-thin hyperbolic.
\end{thm}

\begin{proof}
For $x, y \in \mcP(S)$ define $\mcL(x,y)$ to be the collection of hierarchy paths between $x$ and $y$.  These are connected because each hierarchy path is connected and all contain $x$ and $y$.  By Theorem \ref{hierarchy k-centered} and Lemma \ref{centered to thin} we have that for all $x, y, z \in \mcP(S)$
\begin{equation*}
\mcL(x, y) \subset N_{4*8,900}(\mcL(x,z) \cup \mcL(z,y)).
\end{equation*}
If $d(x,y) \leq 1$ then any hierarchy between $x$ and $y$ is just the edge $\{ xy\}$, so $\mcL(x,y) = \{x, y\}$.  Thus, both conditions of Proposition \ref{subset hyperbolic} are satisfied.  Therefore by applying Proposition \ref{subset hyperbolic} we get $\mcP(S)$ is $2,691,437$-thin hyperbolic.
\end{proof}

\section{Relative Hyperbolicity of Pants Graphs Complexity 3}

In this section we turn our attention to relative pants graphs and their hyperbolicity constant.

\begin{thm}
\label{relative hierarchy k-centered}
Take $S$ such that $\xi(S) = 3$.  The relative 3-archy triangles in $\mcP_{rel}(S)$ are $6,191,300$-centered. 
\end{thm}

\begin{proof} 
Take three pants decompositions of $S$, say $\alpha = \{\alpha_0, \alpha_1, \alpha_2 \}$, $\beta = \{\beta_0, \beta_1, \beta_2\}$, and $\gamma = \{\gamma_0, \gamma_1, \gamma_2\}$.  Form the triangle $\alpha\beta\gamma$ such that each edge in the triangle is a relative 3-archy in $\mcP_{rel}(S)$.  Let $g_{\alpha\beta}$, $g_{\beta\gamma}$, and $g_{\alpha\gamma}$ be the three main geodesics that make up the triangle (which connects $\alpha_0$, $\beta_0$, and $\gamma_0$).  As before in Theorem \ref{hierarchy k-centered}, there are three cases:

\begin{enumerate}
\item All three main geodesics have a curve in common. 
\item Any two of the main geodesics share a curve, but not the third.
\item None of the main geodesics have common curves.
\end{enumerate} 

For the rest of the proof, note that if $v \in \mcC(S)$ is a non-domain separating curve, then $S \backslash v$ has one connected component with positive complexity, so by abuse of notation, we denote this component as $S \backslash v$.  This means that every curve in $\mcC(S)$ not equal to $v$ intersects $S \backslash v$ so we can use the Bounded Geodesic Image Theorem on any geodesic that doesn't contain $v$.  Take two non-domain separating curve $v,w \in \mcC(S)$ such that $v$ and $w$ are disjoint.   Then, because $\xi(S) = 3$, $S \backslash (v \cup w)$ has one connected component with positive complexity, and again we denote this component as $S \backslash (v \cup w)$.   Furthermore, every curve in $\mcC(S)$ not equal to $v$ or $w$ intersects $S \backslash (v \cup w)$, so we may use the Bounded Geodesic Image Theorem for any geodesic that doesn't contain $v$ or $w$.  

Whenever a domain separating curve, $c$, shows up in a relative 3-archy in $\mcP_{rel}(S)$, the section of the relative 3-archy containing $c$ has length $2$.  Therefore, when referring to a curve along a geodesic within a relative 3-archy we will assume it is non-domain separating since this type of curve adds the most length to the relative 3-archy.  This also just makes the proof cleaner.

\textbf{Case 1:} Let $v$ be a vertex where all three main geodesics intersect.  If $v$ is a domain separating curve then each edge of the triangle $\alpha\beta\gamma$ contains the point $p_v$, so the triangle is $0$-centered.  Now assume $v$ is not a domain separating curve.
Let $v_{\alpha\beta}^{-1}$ and $v_{\alpha\beta}^{+1}$ be the curves that are directly before and after $v$ on $g_{\alpha\beta}$.  Similarly define $v_{\alpha\gamma}^{-1}$, $v_{\alpha\gamma}^{+1}$, $v_{\beta\gamma}^{-1}$, and $v_{\beta\gamma}^{+1}$.  Consider the geodesics associated with $v$ in each relative 3-archy edge; in other words, all geodesics in the relative 3-archy that contribute to defining the path where $v$ is a part of every pants decomposition. 

Let $x_{\alpha\beta}$ be the curve in $[v_{\alpha\beta}^{-1}, v_{\alpha\beta}^{+1}]_{S \backslash v}$ that is adjacent to $v_{\alpha\beta}^{-1}$; similarly define $x_{\alpha\gamma}$.  Now connect $\{v_{\alpha\beta}^{-1}, x_{\alpha\beta} \}$ to $\{v_{\alpha\gamma}^{-1}, x_{\alpha\gamma}\}$ with a hierarchy in $\mcP(S \backslash v)$.
Note, to make our notation cleaner, we will refer to this as the hierarchy between $v_{\alpha\beta}^{-1}$ and $v_{\alpha\gamma}^{-1}$; similarly later on we won't necessarily specify the second curve.      
By the Bounded Geodesic Image Theorem the geodesic connecting $v_{\alpha\beta}^{-1}$ and $v_{\alpha\gamma}^{-1}$ in $\mcC(S\backslash v)$ has length at most $2M$.  Now consider any curve, $w$, in the geodesic $[v_{\alpha\beta}^{-1}, v_{\alpha\gamma}^{-1}]_{S \backslash v}$ contained in the hierarchy connecting $\{v_{\alpha\beta}^{-1}, x_{\alpha\beta} \}$ to $\{v_{\alpha\gamma}^{-1}, x_{\alpha\gamma}\}$.  
Assume $w$ is not a domain separating curve in $S$ and let $w^{-1}$ and $w^{+1}$ be the two curves before and after $w$ on $[v_{\alpha\beta}^{-1}, v_{\alpha\gamma}^{-1}]_{S \backslash v}$.  
Then the geodesic connecting $w^{-1}$ to $w^{+1}$ in $\mcC(S \backslash (v \cup w))$ has length at most $4M$ by using the Bounded Geodesic Image Theorem on $[w^{-1}, v_{\alpha\beta}^{-1}]_{S \backslash v} \cup [v_{\alpha\beta}^{-1}, \alpha_0]_S \cup [\alpha_0, v_{\alpha\gamma}^{-1}]_S \cup [v_{\alpha\gamma}^{-1}, w^{+1}]_{S \backslash v}$; note $w$ cannot be on this path because $w$ is distance $1$ from $v$, so if it was anywhere in the path it would be violating the assumption that we have geodesics.  Therefore the hierarchy between $v_{\alpha\beta}^{-1}$ and $v_{\alpha\gamma}^{-1}$ has length at most $8M^2$.  Similarly the hierarchies between $v_{\alpha\gamma}^{+1}$ and $v_{\beta\gamma}^{+1}$, and $v_{\alpha\beta}^{+1}$ and $v_{\beta\gamma}^{-1}$ have length less than $8M^2$.  

Now, make a hierarchy triangle $v_{\alpha\beta}^{+1}v_{\alpha\gamma}^{-1}v_{\beta\gamma}^{+1}$ in $\mcP(S \backslash v)$, see Figure \ref{links are thin} for how this fits in with above.  By Theorem \ref{hierarchy k-centered}, $v_{\alpha\beta}^{+1}v_{\alpha\gamma}^{-1}v_{\beta\gamma}^{+1}$ in $\mcP(S \backslash v)$ is $8,900$ centered, call the point at the center $z$.  
Then by Theorem \ref{hierarchy k-centered} and Lemma \ref{centered to thin}, the hierarchy triangles $v_{\alpha\beta}^{+1}v_{\alpha\beta}^{-1}v_{\alpha\gamma}^{-1}$, $v_{\beta\gamma}^{-1}v_{\beta\gamma}^{+1}v_{\alpha\gamma}^{+1}$, and $v_{\alpha\gamma}^{-1}v_{\alpha\gamma}^{+1}v_{\beta\gamma}^{+1}$ are $35,600$ thin.  Therefore $z$ is at most $124,500$ away from each $[v_*^{+1}, v_{*}^{-1}]_{S \backslash v}$.  This implies that $\{z, v\}$ is at most $124,500$-centered in the relative 3-archy triangle $\alpha\beta\gamma$.  

\begin{figure}
\centering
\includegraphics[trim = {0, 3.4in, 0, 4.1in}, clip, scale = .5]{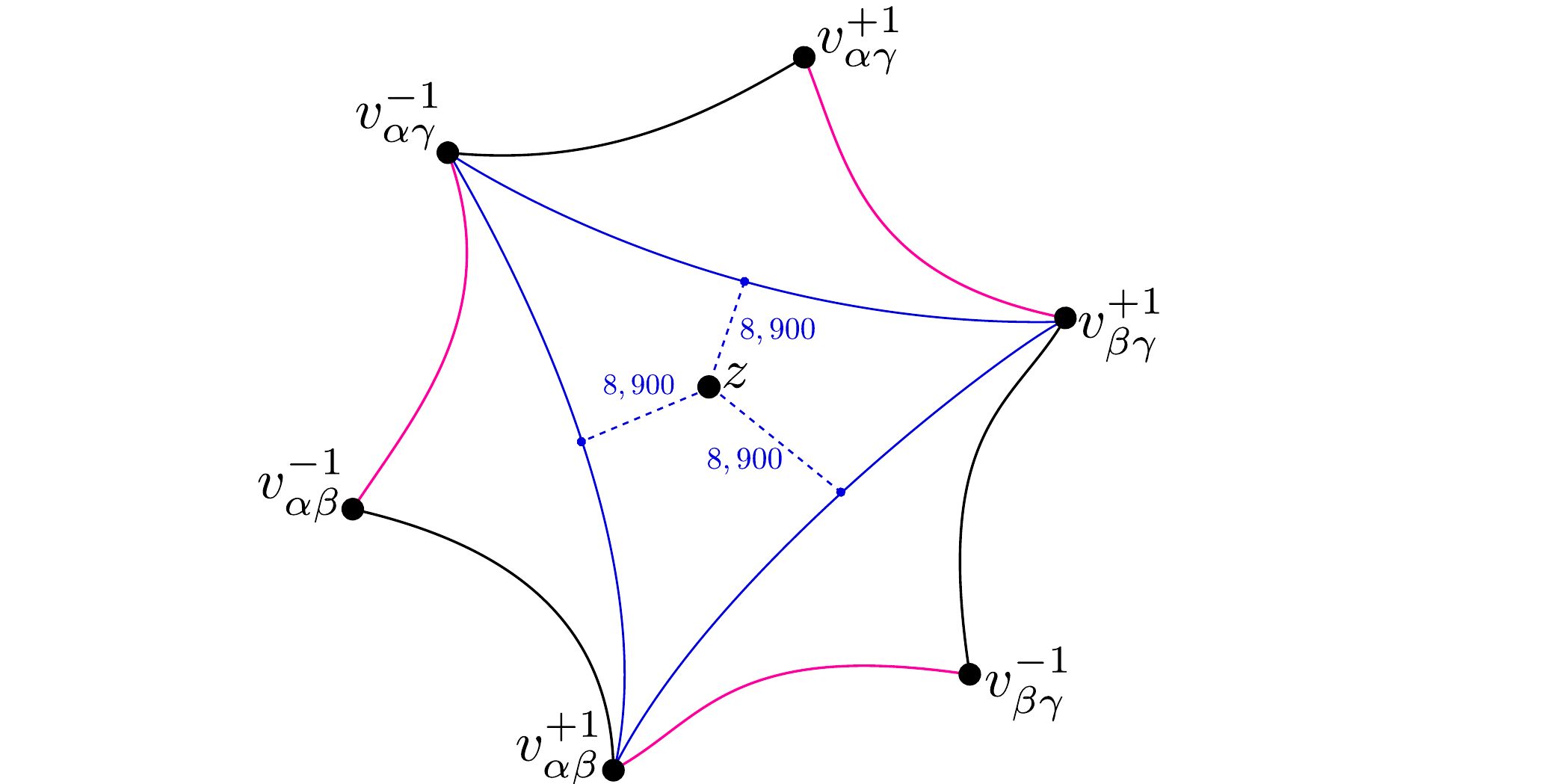}
\caption{All lines above represent the main geodesics in $\mcC(S \backslash v)$ which are a part of a hierarchy.  The black lines come from the geodesics in our original relative 3-archy triangle $\alpha\beta\gamma$.  The pink lines are the ones we added, and the blue lines make the inner triangle that reduces to the case covered by Theorem \ref{hierarchy k-centered}.}
\label{links are thin}
\end{figure}

\textbf{Case 2:} For the same reasons as in Theorem \ref{hierarchy k-centered} case 2, this case can be reduced to case 3.

\textbf{Case 3:} This proceeds with the same strategy as in case 3 of Theorem \ref{hierarchy k-centered}.  By Theorem \ref{curve hyp}, we know the triangle of main geodesics, $g_{\alpha\beta}g_{\beta\gamma}g_{\alpha\gamma}$ in $\mcC(S)$ is $17$-centered.  Let $c$ be the curve that is at the center of this triangle.  Connect $c$ to $g_{\alpha\beta}$, $g_{\beta\gamma}$, and $g_{\alpha\gamma}$ with a geodesic in $\mcC(S)$. 
Define $v_{\alpha \beta}$ to be the vertex in $g_{\alpha \beta}$ which is the least distance to $c$, and similarly define $v_{\alpha \gamma}$ and $v_{\beta \gamma}$.  

Let $c_0$ be the curve directly preceding $c$ in $[v_{\alpha\beta}, c]_S$ and
let $c^{-1}$ be the curve directly preceding $c_0$. Consider a geodesic in $\mcC(S \backslash c_0)$ which connects $c^{-1}$ to $c$, define $c_1$ to be the curve directly preceding $c$ in this geodesic.  We will show $\{c, c_0, c_1\}$ is a center of our relative 3-archy triangle $\alpha\beta\gamma$.  

Let $v_{\beta\gamma}^{-1}$ be the curve before $v_{\beta\gamma}$ in $g_{\beta\gamma}$ and $v_{\beta\gamma}'$ be the curve adjacent to $v_{\beta\gamma}$ in the geodesic contained in the relative 3-archy connecting $\beta$ to $\gamma$ whose domain is $\mcC(S \backslash v_{\beta\gamma}^{-1})$.  Now connect $\{v_{\beta\gamma}, v_{\beta\gamma}^{-1}, v_{\beta\gamma}'\}$ to $\{c, c_0, c_1 \}$ with a relative 3-archy, $H$.  Our goal is to bound the length of $H$.

Using the exact argument as in Theorem \ref{hierarchy k-centered} case 3, for each $w \in [c, v_{\beta\gamma}]_S$ which is non-separating, the geodesic in $H$ whose domain is $\mcC(S\backslash w)$ has length no more than $5M$.  Let $w^{-1}$ and $w^{+1}$ be the curves before and after $w$ in $[c, v_{\beta\gamma}]_S$ and let $[w^{-1}, w^{+1}]_{S \backslash w}$ be the geodesic coming from $H$.  
Take $z \in [w^{-1}, w^{+1}]_{S \backslash w}$ and consider the geodesic in $H$ with domain $\mcC(S \backslash (w \cup z))$.  Define $z^{-1}$ and $z^{+1}$ to be the curves before and after $z$ on $[w^{-1}, w^{+1}]_{S \backslash w}$.  
We will show $[z^{-1}, z^{+1}]_{S \backslash (w \cup z)}$ has length at most $7M$.  Assume towards a contradiction that the length of $[z^{-1}, z^{+1}]_{S \backslash (w \cup z)}$ is greater than $7M$.  Then the path 
$[z^{+1}, w^{+1}]_{S \backslash w} \cup [w^{+1}, v_{\beta\gamma}]_{S} \cup [v_{\beta\gamma}, \gamma_0]_S \cup [\gamma_0, v_{\alpha\gamma}]_S \cup [v_{\alpha\gamma}, c]_S \cup [c, w^{-1}]_S \cup [w^{-1}, z^{-1}]_{S \backslash w}$ 
must contain $z$ or $w$ somewhere, otherwise by the Bounded Geodesic Image Theorem using this path we would get that the length of $[z^{-1}, z^{+1}]_{S \backslash (w \cup z)}$ is at most $7M$.  Since $w$ and $z$ are distance $1$ apart, it doesn't matter which one shows up in the path because we eventually will arise at the same contradiction. Thus, without loss of generality we assume $z$ is in the path (and all other paths considered for this argument).  Then $z$ must be in $[\gamma_0, v_{\alpha\gamma}]_S$ or $[v_{\alpha\gamma}, c]_S$, otherwise there would be a contradiction with the definition of a geodesic or the definition of $c$ or $v_{\beta\gamma}$   Without loss of generality assume $z \in [\gamma_0, v_{\alpha\gamma}]_S$.  Similarly the path 
$[z^{+1}, w^{+1}]_{S \backslash w} \cup [w^{+1}, v_{\beta\gamma}]_{S} \cup [v_{\beta\gamma}, \beta_0]_S \cup [\beta_0, v_{\alpha\beta}]_S \cup [v_{\alpha\beta}, c]_S \cup [c, w^{-1}]_S \cup [w^{-1}, z^{-1}]_{S \backslash w}$
must contain $z$.  Again, the only place $z$ could be, without yielding a contradiction, is in $[v_{\alpha\beta}, c]_S$.  However even here, since $z$ is adjacent to $w$, $w$ is strictly closer than $c$ to the three main geodesics of $\alpha\beta\gamma$ which contradicts our choice of $c$.
Therefore, the length of $[z^{-1}, z^{+1}]_{S \backslash (w \cup z)}$ is at most $7M$. 
Now all that's left to bound is the beginning and end geodesics, i.e. the ones associated to $c$ and $v_{\beta\gamma}$.

Let $y$ be the curve adjacent to $v_{\beta\gamma}$ in $[c, v_{\beta\gamma}]_S$ and let $y'$ be the curve adjacent to $v_{\beta\gamma}$ in the geodesic contained in $H$ whose domain is $\mcC(S \backslash y)$.  Then the very beginning part of $H$ is the hierarchy connecting $\{y, y' \}$ to $\{v_{\beta\gamma}^{-1}, v_{\beta\gamma}' \}$ in $S \backslash v_{\beta\gamma}$.  
We will first bound the length of the geodesic $[y, v_{\beta\gamma}^{-1}]_{S \backslash v_{\beta\gamma}}$.  Assume that the length is more than $5M$.  Then the path $[v_{\beta\gamma}^{-1}, \beta_0]_S \cup [\beta_0, v_{\alpha\beta}]_S \cup [v_{\alpha\beta}, c]_S \cup [c, y]_S$ has to contain $v_{\beta\gamma}$.  By our assumption that the main geodesics on the triangle $\alpha\beta\gamma$ don't intersect, the only part of the path that $v_{\beta\gamma}$ could be on without forming a contraction would be $[v_{\alpha\gamma}, c]_S$.  
The same is true of the path $[v_{\beta\gamma}^{-1}, \beta_0]_S \cup [\beta_0, \alpha_0]_S \cup [\alpha_0, v_{\alpha\gamma}]_S \cup [v_{\alpha\gamma}, c]_S \cup [c, y]_S$, where $v_{\beta\gamma}$ would have to be in $[v_{\alpha\gamma}, c]_S$.  However, then we could take $v_{\beta\gamma}$ to be the center of the main geodesic triangle which would give strictly smaller lengths to each of the sides, contradicting our choice of $c$.  Therefore, $[y, v_{\beta\gamma}^{-1}]_{S \backslash v_{\beta\gamma}}$ has length at most 5M.  

Now take $w \in [y, v_{\beta\gamma}^{-1}]_{S \backslash v_{\beta\gamma}}$ and let $w^{-1}$ and $w^{+1}$ be the curves that come directly before and after $w$ in $[y, v_{\beta\gamma}^{-1}]_{S \backslash v_{\beta\gamma}}$.  We want to bound the length of $[w^{-1}, w^{+1}]_{S \backslash (v_{\beta\gamma} \cup w)}$.  Assume the length is greater than $7M$.  
Then the path $[w^{+1}, v_{\beta\gamma}^{-1}]_{S \backslash v_{\beta\gamma}} \cup [v_{\beta\gamma}^{-1}, \beta_0]_S \cup [\beta_0, v_{\alpha\beta}]_S \cup [v_{\alpha\beta}, c]_S \cup[c, y]_S \cup [y, w^{-1}]_{S \backslash v_{\beta\gamma}}$ must contain $w$ or $v_{\beta\gamma}$.  The only two places this could happen without raising a contradiction is in $[\beta_0, v_{\alpha\beta}]_S$ or $[v_{\alpha\beta}, c]_S$.  Again, whether we assume $w$ or $v_{\beta\gamma}$ is in the path doesn't matter since we will arrive at the same contradiction, hence we can assume without loss of generality $w$ is always on the path.  Therefore, assume $w \in [v_{\alpha\beta}, c]_S$.
Similarly, $w$ is contained in the path $[w^{+1}, v_{\beta\gamma}^{-1}]_{S \backslash v_{\beta\gamma}} \cup [v_{\beta\gamma}^{-1}, v_{\beta\gamma}^{+1}]_{S \backslash v_{\beta\gamma}} \cup [v_{\beta\gamma}^{+1}, \gamma_0] \cup [\gamma_0, v_{\alpha\gamma}]_S \cup [v_{\alpha\gamma}, c]_S \cup[c, y]_S \cup [y, w^{-1}]_{S \backslash v_{\beta\gamma}}$, where $w \in [\gamma_0, v_{\alpha\gamma}]_S$ since anywhere else in the path would lead to a contradiction as explained previously.  Note if $w \in [v_{\alpha\gamma}, c]_S$ then since $w$ is disjoint from $v_{\beta\gamma}$ and that $w \in [v_{\alpha\beta}, c]_S$, we could make a shorter path to each of the three sides on the main geodesic triangle and then $v_{\beta\gamma}$ would be the center of the triangle, contradicting our choice of $c$.  
The path $[w^{+1}, v_{\beta\gamma}^{-1}]_{S \backslash v_{\beta\gamma}} \cup [v_{\beta\gamma}^{-1}, \beta_0]_S \cup [\beta_0, \alpha_0]_S \cup [\alpha_0, v_{\alpha\gamma}]_S \cup [v_{\alpha\gamma}, c]_S \cup[c, y]_S \cup [y, w^{-1}]_{S \backslash v_{\beta\gamma}}$ has to contain $w$ as well.  No matter where $w$ is on this path is creates a contradiction - either with the definition of $c$, with the we have a geodesic, or with the assumption the main geodesics do not share any curves.  Consequently, $[w^{-1}, w^{+1}]_{S \backslash (v_{\beta\gamma} \cup w)}$ must have length at most $7M$.  Note that this argument also works when $w = y$ or $w = v_{\beta\gamma}^{-1}$, which gives a length bound on the geodesic in $H$ whose domain is $\mcC(S \backslash (v_{\beta\gamma} \cup y))$ or $\mcC(S \backslash (v_{\beta\gamma} \cup v_{\beta\gamma}^{-1}))$, respectively.  

Let $x$ be the curve adjacent to $c$ in $[v_{\beta\gamma}, c]_S$ and $x'$ be the last curve adjacent to $c$ in the geodesic from the hierarchy whose domain is $\mcC(S \backslash x)$.  First, the geodesic $[c_0, x]_{S \backslash c}$ has length no more than $4M$ by the Bounded Geodesic Image Theorem applied to $[c_0, v_{\alpha\beta}]_S \cup [v_{\alpha\beta}, \beta_0]_S \cup [\beta_0, v_{\beta\gamma}]_S \cup [v_{\beta\gamma}, x]_S$, which doesn't contain $c$ because if it did we would get a contradiction on the definition of $c$. Now take any curve $w \in [c_0, x]_{S \backslash c}$ and define $w^{-1}$ and $w^{+1}$ as before.  Then the path 
$[w^{+1}, x]_{S \backslash c} \cup [x, v_{\beta \gamma}]_S \cup [v_{\beta\gamma}, \beta_0]_S \cup [\beta_0, v_{\alpha\beta}]_S \cup [v_{\alpha\beta}, c_0]_S \cup [c_0, w^{-1}]_{S \backslash c}$ cannot contain $w$ because $w$ is adjacent to $c$ so if any geodesic making up the path contained $w$ it would either contradict that it is a geodesic or that $c$ is minimal distance from the main geodesics of the triangle $\alpha\beta\gamma$.  Hence, applying the Bounded Geodesic Image Theorem to the path we get that $[w^{-1}, w^{+1}]_{S \backslash (c \cup w)}$ has length no more than $6M$.  This leaves bounding the lengths of the geodesics connecting $c_1$ to the second vertex of $[c_0, x]_{S \backslash c}$ and $x'$ to the penultimate vertex of $[c_0, x]_{S \backslash c}$.  By a similar argument using the Bounded Geodesic Image Theorem each of these geodesics have length at most $6M$.  Therefore, putting all the length bounds together we get that the relative 3-archy connecting $\{v_{\beta\gamma}, v_{\beta\gamma}^{-1}, v_{\beta\gamma}'\}$ to $\{c, c_0, c_1 \}$ has length at most $16*5M*7M + (4M-1)*6M+12M + (5M+1)*7M = 6,191,300$

Similarly $\{c, c_0, c_1\}$ is length at most $6,191,300$ from the other two sides of the triangle $\alpha\beta\gamma$.  Therefore, the relative 3-archy triangle $\alpha\beta\gamma$ is $6,191,300$-centered.  
\end{proof} 

\begin{thm}
\label{main thm 2}
For a surface $S$ such that $\xi(S) =3$, $\mcP_{rel}(S)$ is $1,607,425,314$-thin hyperbolic.
\end{thm}

\begin{proof}
For $x, y \in \mcP_{rel}(S)$ define $\mcL(x,y)$ to be the collection of relative 3-archy paths between $x$ and $y$.  These are connected because each relative 3-archy path is connected and all the relative 3-archies in $\mcL(x, y)$ contain $x$ and $y$.  By Theorem \ref{relative hierarchy k-centered} and Lemma \ref{centered to thin} we have that for all $x, y, z \in \mcP_{rel}(S)$
\begin{equation*}
\mcL(x, y) \subset N_{4*6,191,300}(\mcL(x,z) \cup \mcL(z,y)).
\end{equation*}
If $d(x,y) \leq 1$ then any relative 3-archy between $x$ and $y$ is just the edge $\{ xy\}$, so $\mcL(x,y) = \{x, y\}$.  We now have both conditions of Proposition \ref{subset hyperbolic} satisfied.  Therefore by applying Proposition \ref{subset hyperbolic} we get that $\mcP_{rel}(S)$ is $1,607,425,314$-thin hyperbolic.
\end{proof}

\bibliographystyle{plain}
\bibliography{mybib}

\begin{bibdiv}
\begin{biblist}

\bib{ABC}{incollection}{
      author={Alonso, J.~M.},
      author={Brady, T.},
      author={Cooper, D.},
      author={Ferlini, V.},
      author={Lustig, M.},
      author={Mihalik, M.},
      author={Shapiro, M.},
      author={Short, H.},
       title={Notes on word hyperbolic groups},
        date={1991},
   booktitle={Group theory from a geometrical viewpoint ({T}rieste, 1990)},
   publisher={World Sci. Publ., River Edge, NJ},
       pages={3\ndash 63},
        note={Edited by Short},
      review={\MR{1170363}},
}

\bib{ATW}{unpublished}{
      author={Aougab, Tarik},
      author={Taylor, Samuel~J.},
      author={Webb, Richard C.~H.},
       title={Effective masur-minsky distance formulas and applications to
  hyperbolic 3-manifolds},
        note={preprint},
}

\bib{BF}{article}{
      author={Brock, Jeffrey},
      author={Farb, Benson},
       title={Curvature and rank of {T}eichm\"uller space},
        date={2006},
        ISSN={0002-9327},
     journal={Amer. J. Math.},
      volume={128},
      number={1},
       pages={1\ndash 22},
}

\bib{BM-WPisom}{article}{
      author={Brock, Jeffrey},
      author={Margalit, Dan},
       title={Weil-{P}etersson isometries via the pants complex},
        date={2007},
        ISSN={0002-9939},
     journal={Proc. Amer. Math. Soc.},
      volume={135},
      number={3},
       pages={795\ndash 803},
         url={https://doi.org/10.1090/S0002-9939-06-08577-7},
      review={\MR{2262875}},
}

\bib{BM}{article}{
      author={Brock, Jeffrey},
      author={Masur, Howard},
       title={Coarse and synthetic {W}eil-{P}etersson geometry: quasi-flats,
  geodesics and relative hyperbolicity},
        date={2008},
        ISSN={1465-3060},
     journal={Geom. Topol.},
      volume={12},
      number={4},
       pages={2453\ndash 2495},
}

\bib{Bow}{article}{
      author={Bowditch, Brian~H.},
       title={Uniform hyperbolicity of the curve graphs},
        date={2014},
        ISSN={0030-8730},
     journal={Pacific J. Math.},
      volume={269},
      number={2},
       pages={269\ndash 280},
}

\bib{Brock-WPtoPants}{article}{
      author={Brock, Jeffrey~F.},
       title={The {W}eil-{P}etersson metric and volumes of 3-dimensional
  hyperbolic convex cores},
        date={2003},
        ISSN={0894-0347},
     journal={J. Amer. Math. Soc.},
      volume={16},
      number={3},
       pages={495\ndash 535},
         url={https://doi.org/10.1090/S0894-0347-03-00424-7},
      review={\MR{1969203}},
}

\bib{Brock-WPtrans}{article}{
      author={Brock, Jeffrey~F.},
       title={Weil-{P}etersson translation distance and volumes of mapping
  tori},
        date={2003},
        ISSN={1019-8385},
     journal={Comm. Anal. Geom.},
      volume={11},
      number={5},
       pages={987\ndash 999},
         url={https://doi.org/10.4310/CAG.2003.v11.n5.a6},
      review={\MR{2032506}},
}

\bib{BellWebb}{article}{
      author={{Bell}, Mark~C.},
      author={{Webb}, Richard C.~H.},
       title={{Polynomial-time algorithms for the curve graph}},
        date={2016-09},
     journal={arXiv e-prints},
       pages={arXiv:1609.09392},
      eprint={1609.09392},
}

\bib{Harvey}{incollection}{
      author={Harvey, W.~J.},
       title={Boundary structure of the modular group},
        date={1981},
   booktitle={Riemann surfaces and related topics: {P}roceedings of the 1978
  {S}tony {B}rook {C}onference ({S}tate {U}niv. {N}ew {Y}ork, {S}tony {B}rook,
  {N}.{Y}., 1978)},
      series={Ann. of Math. Stud.},
      volume={97},
   publisher={Princeton Univ. Press, Princeton, N.J.},
       pages={245\ndash 251},
      review={\MR{624817}},
}

\bib{HPW}{article}{
      author={Hensel, Sebastian},
      author={Przytycki, Piotr},
      author={Webb, Richard C.~H.},
       title={1-slim triangles and uniform hyperbolicity for arc graphs and
  curve graphs},
        date={2015},
        ISSN={1435-9855},
     journal={J. Eur. Math. Soc. (JEMS)},
      volume={17},
      number={4},
       pages={755\ndash 762},
         url={https://doi.org/10.4171/JEMS/517},
      review={\MR{3336835}},
}

\bib{Huang}{article}{
      author={Huang, Zheng},
       title={On asymptotic {W}eil-{P}etersson geometry of {T}eichm\"{u}ller
  space of {R}iemann surfaces},
        date={2007},
        ISSN={1093-6106},
     journal={Asian J. Math.},
      volume={11},
      number={3},
       pages={459\ndash 484},
         url={https://doi.org/10.4310/AJM.2007.v11.n3.a5},
      review={\MR{2372726}},
}

\bib{Irmer}{article}{
      author={Irmer, Ingrid},
       title={Stable lengths on the pants graph are rational},
        date={2015},
        ISSN={1076-9803},
     journal={New York J. Math.},
      volume={21},
       pages={1153\ndash 1168},
         url={http://nyjm.albany.edu:8000/j/2015/21_1153.html},
      review={\MR{3425638}},
}

\bib{MMII}{article}{
      author={Masur, H.~A.},
      author={Minsky, Y.~N.},
       title={Geometry of the complex of curves. {II}. {H}ierarchical
  structure},
        date={2000},
        ISSN={1016-443X},
     journal={Geom. Funct. Anal.},
      volume={10},
      number={4},
       pages={902\ndash 974},
         url={https://doi.org/10.1007/PL00001643},
      review={\MR{1791145}},
}

\bib{Webb}{thesis}{
      author={Webb, Richard~C.H.},
       title={Effective geometry of curve graphs},
        type={Ph.D. Thesis},
        date={2014},
}

\end{biblist}
\end{bibdiv}

\noindent {\it Email:}\\
aweber@math.brown.edu

\end{document}